\documentclass[11pt]{amsart}
\usepackage{amssymb,amsmath,cite,hyperref}
\usepackage[usenames]{color} \newtheorem{theorem}{Theorem}[section]
\newtheorem{proposition}[theorem]{Proposition}
\newtheorem{lemma}[theorem]{Lemma}

\newtheorem{remark}[theorem]{Remark}
\newtheorem{problem}[theorem]{Problem}
\newtheorem{question}[theorem]{Question}
\newtheorem{counterexample}[theorem]{Counterexample}

\DeclareMathOperator{\link}{Link}

\title{On vanishing patterns in $j$-strands of edge ideals}
\author{Abed Abedelfatah and Eran Nevo}
\address{Einstein Institute of Mathematics, The Hebrew University of Jerusalem, Jerusalem, Israel}
\email{abedelfatah@gmail.com}
\address{Einstein Institute of Mathematics, The Hebrew University of Jerusalem, Jerusalem, Israel}
\email{nevo@math.huji.ac.il}
\thanks{Research of both authors was partially supported by Israel Science Foundation grants ISF-805/11 and ISF-1695/15.
}
\begin{document}
\maketitle
\begin{abstract}
We consider two problems regarding vanishing patterns in the Betti table of edge ideals $I$ in polynomial algebra $S$. First, we show that the $j$-strand is connected if $j=3$ (for $j=2$ this is easy and known), and give examples where the $j$-strand is not connected for any $j>3$.
Next, we apply our result on strand connectivity to establish the subadditivity conjecture for edge ideals, $t_{a+b}\leq t_a+t_b$, in case $b=2,3$ (the case $b=1$ is known). Here $t_i$ stands for the maximal shifts in the minimal free $S$-resolution of $S/I$.
\end{abstract}
\section{Introduction}

Let $S=K[x_1,\ldots,x_n]$ denote the polynomial ring with $n$ variables over the fixed field $K$, graded by setting $\deg(x_i)=1$ for each variable.
Since Hilbert's syzygy theorem, minimal free resolutions of graded finitely generated $S$-modules, and particularly their graded Betti numbers, became central invariants of study in Commutative Algebra, with applications in other areas, e.g. in algebraic geometry, hyperplane arrangements, and combinatorics.
Many important invariants of such $S$-modules are determined just by the \emph{vanishing pattern} of the graded Betti numbers, namely which ones are zero and which are nonzero, e.g. the regularity, projective dimension etc.
Restricting to $S$-modules $S/I$ for monomial ideals $I$, and particularly to edge ideals, makes combinatorial, and particularly graph theoretical, tools available. This perspective proved to be very fruitful in recent decades; see the recent textbooks \cite{Herzog-Hibi, Peeva} for more background and references.

In this paper we consider the following two problems on vanishing patterns in the Betti table of edge ideals; we completely resolve the first and use it to partially resolve the second.

(I)
In a private communication, Aldo Conca asked us the following question, based on computer experiments; see also Whieldon [\citen{WG-jump},Question 7.1(2)]:

\begin{question}\label{q:Aldo}
Is the Betti diagram of any monomial ideal $I$ generated in degree 2, for any $j\geq 2$, $j$-strand connected? I.e. if $j\geq2$, $\beta_{i,j}(I)$ and $\beta_{i+k,j+k}(I)$ are both non-zero, where $k>0$ and $i\geq 0$, then $\beta_{i+m,j+m}(I)\neq0$ for all $0\leq m\leq k$.
\end{question}
The answer to this question, over any field, is ``Yes'' when $j=2$.
This follow from the fact that if $\cdots\rightarrow F_1\rightarrow F_0\rightarrow I\rightarrow 0$ is the minimal graded free resolution of $I$ and $r$ is the minimum degree of the generators of $F_i$, then the minimum degree of the generators of $F_{i+1}$ is at least $r+1$. So if $(\beta_{0,2}(I),\beta_{1,3}(I),\dots)$ is the first strand and $\beta_{i,i+2}(I)=0$ for some $i\geq0$, then $\beta_{t,t+2}(I)=0$ for all $t\geq i$.

We show that the answer is ``Yes" for $j=3$ and ``No" for any $j>3$:
\begin{theorem}\label{thm:I}
Over any field, any monomial ideal generated in degree 2 is $j$-strand connected for $j=3$ or $2$.

For any $j>3$, there is a monomial ideal generated in degree 2 which is not $j$-strand connected, over any field.
\end{theorem}
Without the assumption on generation in degree $2$, easier examples in Remark \ref{rem} show that the answer is ``No'' for any $j>2$.

By polarization, we can reduce the problem to squarefree monomial ideals. Using Hochster formula, we answer Question \ref{q:Aldo} by topological combinatorics arguments; see Section \ref{sec:Proofs}.

Theorem \ref{thm:I} can be visualized as a \emph{vanishing pattern} on the Betti table of $I$, where in the $(i,j)$ entry we put $X$ if $\beta_{i,i+j}(I)\neq 0$ and $0$ otherwise; call it the \emph{vanishing table} of $I$. Then our result says that if a monomial ideal $I$ is generated in degree $2$ then its vanishing table has no subsequence with internal zeros $(X,0,\ldots,0,X)$ in any of the first two rows.
For other recent results on other vanishing patterns, see e.g.~\cite{Oscar, Torrente-Varbaro}.



%

(II)
Second, we consider the subadditivity problem for edge ideals.
Given a graded ideal $I$ in $S$ let $t_i$ denote the maximal shifts in the minimal graded free $S$-resolution of $S/I$, namely
\[t_i=t_i(S/I):= \max(j:\ \beta_{i,j}(S/I)\neq 0).\]
The subadditivity relation
\[(*) \ \ t_{a+b}\leq t_a + t_b\]
was proved under certain conditions on $I$ or for certain values of $a$ and $b$, e.g. in \cite{Conca-Sub, Eisenbud-Sub, Herzog-Srinivasan, Srinivasan-New}, and is conjectured to hold under other conditions on $I$ for all values $a$ and $b$ \cite[Conjecture 6.4]{Conca-Sub}.
While counterexamples to (*) for general graded ideals are indicated in \cite[Section 6.1]{Conca-Sub}, no counterexamples to (*) are known for monomial ideals.
When $I$ is generated by monomials, Herzog and Srinivasan \cite[Corollary 4]{Herzog-Srinivasan} proved (*) for $b=1$, which was proved earlier for edge ideals in \cite[Theorem 4.1]{Oscar}.

\begin{theorem}\label{thm:II}
For any edge ideal over any field, the subadditivity relation (*) holds for $b=1,2,3$ and any natural number $a$.
\end{theorem}
The proof of the case $b=3$ uses the connectivity of the $3$-strand, see Theorem \ref{thm:I}.
Topological combinatorics arguments are used here too; see Section \ref{sec:App}.

\section{Preliminaries }
Fix a field $K$.
Let $S=K[x_1,\dots,x_n]$ be the graded polynomial ring with $\deg(x_i)=1$ for all $i$, and $M$ be a graded $S$-module. The integer $\beta_{i,j}^S(M)=\dim_KTor_i^S(M,K)_j$ is called the $(i,j)$\emph{th graded Betti number of} $M$. Note that if $I$ is a graded ideal of $S$, then $\beta_{i+1,j}^S(S/I)=\beta_{i,j}^S(I)$ for all $i,j\geq 0$.

For a simplicial complex $\Delta$ on the vertex set $\Delta_0=[n]=\{1,\dots,n\}$, its \emph{Stanley-Reisner ideal} $I_{\Delta}\subset S$ is the ideal generated by the squarefree monomials $x_F=\prod_{i\in F}x_i$ with $F\notin \Delta$, $F\subset [n]$. A simplicial complex is called \emph{flag} if its Stanley-Reisner ideal is generated by squarefree monomials of degree two. Flag simplicial complexes are closely related to simple graphs. Let $G$ be a simple graph on the set $[n]$ and denote by $E(G)$ the set of its edges. We define the \emph{edge ideal} of $G$ to be the ideal $$I(G)=\langle x_ix_j:~\{i,j\}\in E(G)\rangle\subset S.$$ So if $\Delta$ is a flag simplicial complex and $H$ is the graph of minimal non-faces of $\Delta$, then $I_\Delta=I(H)$.

For $W\subset V$, we write $$\Delta[W]=\{F\in\Delta~:~F\subset W\}$$ for the induced subcomplex of $\Delta$ on $W$. We denote by $\beta_i(\Delta)=\dim_K \widetilde{H}_i(\Delta;K)$ the dimension of the $i$-th reduced homology group of $\Delta$ with coefficients in $K$. Let $F$ be a face of $\Delta$. The \emph{link} of $F$ in $\Delta$ is the
following simplicial complex: $$\link_\Delta F=\{F'~|~F'\cup F\in \Delta~\mathrm{and}~F'\cap F=\emptyset\}.$$
The \emph{Alexander dual complex} of $\Delta$ is $$\Delta^{\vee}=\{[n]\setminus F~|~F\notin \Delta,\ F\subseteq [n]\}.$$
The following result is known as Hochster's formula for graded Betti numbers.

\begin{theorem}[Hochster]
Let $\Delta$ be a simplicial complex on $[n]$. Then
$$\beta_{i,j}(I_{\Delta})=\sum_{W\subset[n],~|W|=j}\beta_{j-i-2}(\Delta[W])$$ for all $i\geq0$ and $j\geq i+1$.
\end{theorem}

From Hochster's formula, if we denote $r=j-i-2$ then we obtain the following equivalent topological version of Question \ref{q:Aldo}.

\begin{question}\label{q2:Aldo}
Let $r\geq0$ and $\Delta$ is a flag complex. Assume that $\beta_r(\Delta[W])\neq 0$ and $\beta_r(\Delta[W'])= 0$ for any $W'$ with $|W'|=|W|-1$. Does it follow that $\beta_r(\Delta[W"])= 0$ for any $W"$ with $|W"|<|W|$?
\end{question}

Eagon and Reiner \cite{eagon} introduced a variant of Hochster's formula that uses Alexander duality.
\begin{theorem}[\cite{eagon}]\label{alex}
Let $\Delta$ be a simplicial complex on $[n]$. Then $$\beta_{i,j}(I_\Delta)=\sum_{F\in \Delta^{\vee},~|F|=n-j}\beta_{i-1}(\link_{\Delta^{\vee}}F)$$ for all $i\geq0$ and $j\geq i+1$.
\end{theorem}

By Theorem \ref{alex}, we obtain the following equivalent version of Question \ref{q2:Aldo} using the links of faces of the Alexander dual.

\begin{question}
Let $r\geq1$ and $\Delta$ is a flag complex. Assume that $\beta_r(\link_{\Delta^{\vee}}F)\neq 0$ and $\beta_{r-1}(\link_{\Delta^{\vee}}F')= 0$ for any $F'$ with $|F'|=|F|+1$. Does it follow that $\beta_{r-k}(\link_{\Delta^{\vee}}F")= 0$ for any $F"$ with $|F"|=|F|+k$ and $k\geq 1$?
\end{question}

If $\Delta_1$ and $\Delta_2$ are two subcomplexes of $\Delta$ such that $\Delta=\Delta_1\cup \Delta_2$, then there is a long exact sequence of reduced homologies (with $K$-coefficients), called the \emph{Mayer-Vietoris sequence}
$$\cdots \rightarrow \widetilde{H}_i(\Delta_1\cap\Delta_2)\rightarrow \widetilde{H}_i(\Delta_1)\oplus \widetilde{H}_i(\Delta_2)\rightarrow \widetilde{H}_i(\Delta)\rightarrow \widetilde{H}_{i-1}(\Delta_1\cap\Delta_2)\rightarrow\cdots$$

Using the Mayer-Vietoris sequence, Fern{\'a}ndez-Ramos and Gimenez proved the following, from which the case $b=1$ in Theorem~\ref{thm:II} readily follows.
\begin{lemma}(\cite[Theorem 2.1]{Oscar})\label{lem:corner}
For an edge ideal $I=I(G)$, over any field, if $\beta_{i,j}(S/I)=0=\beta_{i,j+1}(S/I)$ then $\beta_{i+1,j+2}(S/I)=0$.
\end{lemma}

\section{Strand connectivity}\label{sec:Proofs}
We begin with the following remark.
\begin{remark}\label{rem}
Without the assumption on generation in degree 2 in Question \ref{q:Aldo}, easier examples show that the answer in ``No'' for any $j>2$. Just take the Stanley-Reisner ideal of the join of the boundary of a $(j-1)$-simplex with the barycentric subdivision of this boundary complex. For $j=3$ this is the join of a $3$-cycle and a $6$-cycle.
\end{remark}
\begin{proof}
The only subsets of vertices on which the induced complex has nonzero $(j-2)$-homology are those of each of the two components of the join.
\end{proof}

Next, we give a counterexample to Question \ref{q2:Aldo} for all $r\geq2$, thus also to Question \ref{q:Aldo} for all $j\geq 4$.

\begin{counterexample}\label{ex}
Fix $i\ge 2$.
Let $S$ be a flag sphere of dimension $i$ that contains a subset $A$ of vertices of size $|A|=2i+4$. Denote by $G$ the 1-skeleton of $S$. There exists such $S$ and $A$ with $d(a,b)\geq3$ for all $a,b\in A$, where $d(\cdot,\cdot)$ is the graph metric in $G$.
Let $O$ be an octahedral sphere of dimension $i+1$ on the vertices of $A$. Let $a$ and $b$ be opposite vertices in $O$ and $\Delta=S\cup O$. Then $\beta_i(\Delta)$ and $\beta_i(\Delta[A-\{a,b\}])$ are both non-zero but $\beta_i(\Delta-x)=0$ for all $x\in \Delta_0$.
\end{counterexample}

\begin{proof}
Since $d(a,b)\geq3$ for all $a,b\in A$, it follows that $\Delta$ is a flag complex. The complex $S\cap O$ is the complex of isolated vertices $A$ (denote it also by $A$, abusing notation). By the construction, we have $\widetilde{H}_i(A)=\widetilde{H}_i(O)=\widetilde{H}_{i-1}(A)=0$ (note that $i-1>0$), and $\widetilde{H}_i(S)\neq 0$. Using Mayer-Vietoris sequence $$\cdots \rightarrow \widetilde{H}_i(A)\rightarrow \widetilde{H}_i(S)\oplus \widetilde{H}_i(O)\rightarrow \widetilde{H}_i(\Delta)\rightarrow \widetilde{H}_{i-1}(A)\rightarrow\cdots$$
we get $\widetilde{H}_i(\Delta)\neq0$.

Since $\Delta[A-\{a,b\}]$ is an octahedral sphere of dimension $i$, it follows that $\widetilde{H}_i(\Delta[A-\{a,b\}])\neq0$. Consider a vertex $x\in \Delta_0$.

If $x\in \Delta_0-A$ then $\Delta-x=O\cup(S-x)$ and $O\cap(S-x)=A$. Using the fact that $\widetilde{H}_i(S-x)=0$ and the Mayer-Vietoris sequence $$\cdots \rightarrow \widetilde{H}_i(A)\rightarrow \widetilde{H}_i(S-x)\oplus \widetilde{H}_i(O)\rightarrow \widetilde{H}_i(\Delta-x)\rightarrow \widetilde{H}_{i-1}(A)\rightarrow\cdots$$ we get $\beta_i(\Delta-x)=0$.

If $x\in A$ then $\Delta-x=(S-x)\cup(O-x)$ and $(S-x)\cap(O-x)=A-x$. Note that $O-x$ is an $(i+1)$-ball and so $\widetilde{H}_i(O-x)=0$. Again, the Mayer-Vietoris sequence implies $\beta_i(\Delta-x)=0$.
\end{proof}

In the following theorem, we prove Question \ref{q2:Aldo} for $r=1$, thus complete the proof of Theorem~\ref{thm:I}.
\begin{theorem}\label{strand2}
Let $\Delta$ be a flag simplicial complex. Assume that $\beta_1(\Delta[W])\neq 0$ and $\beta_1(\Delta[W'])= 0$ for any $W'$ with $|W'|=|W|-1$. Then $\beta_1(\Delta[W"])= 0$ for any $W"$ with $|W"|<|W|$.
\end{theorem}
\begin{proof}
Step 1: We claim $\Delta[W]$ is an induced cycle.

Proof: as $\beta_1(\Delta[W])\neq 0$, there is a nontrivial cycle $C \subseteq \Delta[W]$. A chord $e$ in $C$ splits $C$ into two cycles with intersection $e$; at least one of these two cycles is nontrivial, else $C$ would be trivial. Thus, we may assume $C$ is an induced cycle, i.e. chordless. It is left to show that $V(C)=W$. If not, then there is a subset $W'$ of size $|W'|=|W|-1$ s.t. $V(C)\subseteq W'\subseteq W$, so $C$ is nontrivial in $\Delta[W']$, a contradiction to $\beta_1(\Delta[W'])=0$. Denote $C:=\Delta[W]$.

Step 2: Assume by contradiction there is $W"$ with  $|W"|<|W|-1$ and $\beta_1(\Delta[W"])\neq 0$. We may assume $C":=\Delta[W"]$ is an induced cycle (see step 1); and $|W"|\geq 4$ as $\Delta$ is flag.

Look on the vertices of $C"$ in cyclic order.
We claim that between any two vertices in $W\cap W"$ there is a vertex in $W"-W$. In particular (using $|W"|\geq 4$), there are two vertices $a,b\in W"-W$ which are non-neighbors in $C"$.
In case $W\cap W"$ is nonempty, we can choose $a,b$ so that they have a common neighbor in $W\cap W"$; and so we choose.

Proof:
let $x,y\in W\cap W"$ be consecutive in the cyclic order induced on $W\cap W"$ from $C"$, and assume by contradiction that $xy$ is an edge in $C"$. As $C$ is induced, $xy$ must be an edge of $C$ as well.
This leads to a contradiction, as follows: in $W\cup W"$, remove vertices from $W-W"$ to obtain a subset $W'$ of size $|W|-1$. By assumption, $\beta_1(\Delta[W'])=0$, thus the cycle $C"$ is trivial in $\Delta[W']$; in particular there is a triangle $xyz\in \Delta[W']$. However, $z\notin W"$ and $z\notin W$ as both $C$ and $C"$ are induced cycles of length $\geq 4$, a contradiction.

Step 3:
as $|W|\geq 6$, there exist $x,y,z\in C=\Delta[W]$ such that $xy$ is an edge and $z$ not a neighbor of either of $x,y$. Thus, $C-\{x,y,z\}$ is the union of two disjoint nonempty paths, denoted $P$ and $Q$. As $W':=P\cup Q\cup \{a,b\}$
(see Step 2 for who $a,b$ are)
has size $|W|-1$, by assumption $\beta_1(\Delta[W'])=0$.

Note that in the suspension of $P\cup Q$ by $a,b$ any cycle of the form $(a,P',b,Q',a)$, where $P'\subseteq P$ and $Q'\subseteq Q$ are nonempty subpaths, is nontrivial. Thus, the following claim finishes the proof:

(**) There exists a choice of $x,y,z\in C$ as above such that there are vertices $p_a,p_b\in P$ (they may be equal) and
$q_a,q_b\in Q$ such that the four edges $ap_a,aq_a,bp_b,bq_b\in \Delta$ exist.

Step 4: to prove (**), we first claim that each of $a$ and $b$ has at least 6 neighbors in $W$.

Proof: for any $W'$ of size $|W'|=|W|-1$ where $W"\subseteq W'\subseteq W"\cup W$, the induced cycle $C"$ is trivial in $\Delta[W']$, thus each of its edges belong to some triangle, so each of $a,b$ has a neighbor in $W'-W"$.

Recall $|W"|\geq 4$.
If $W\cap W"=\emptyset$, we throw at least 5 vertices of $W-W"$ ($=W$) from $W\cup W"$ to obtain $W'$,
so each of $a,b$ has at least 6 neighbors in $W-W"$;
if $|W\cap W"|=1$ we throw at least 4 vertices of $W-W"$, so each of $a,b$ has at least 5 neighbors in $W-W"$ and another neighbor is in $W\cap W"$.
Thus assume $|W\cap W"|>1$.
The only case when we throw less than 4 vertices of $W-W"$ is when $a,b$ are the only vertices of $W"-W$, in which case we throw 3 vertices of $W-W"$ and both $a$ and $b$ have 2 neighbors in $W\cap W"$ (and $C"$ is a $4$-cycle); so again each of $a$ and $b$ has at least 6 neighbors in $W$, as claimed.

Step 5, proof of (**): fix an orientation on $C$.
Denote some 6 neighbors of $a$ as guaranteed in Step 4 by $(a_1,a_2,a_3,a_4,a_5,a_6)$ in cyclic order on $C$, and similarly $(b_1,b_2,b_3,b_4,b_5,b_6)$ for $b$ (possibly some $a_i=b_j$). For a vertex $v\in C$ denote by $v'$ the vertex right after it on $C$ in the cyclic order.
Tentatively, let $z=a_2, x=a_5, y=a_5'$; so $a$ has a neighbor both in $P$ and $Q$. If $b$ does not have a neighbor in both $P$ and $Q$, one of the following 2 cases must occur:


Case A: there are (at least) 3 $b_j$'s in the segment $P$, $a_2(P)a_5$; by relabeling $(b_1,\ldots,b_6)$ by a cyclic permutation we may assume they are $b_1,b_2,b_3$.
Change the tentative choice to $z=b_2, x=a_6, y=a_6'$.
Then $a_2,b_1$ are in (say) $P$ and $b_3,a_5$ are in $Q$.

Case B: there are (at least) 3 $b_j$'s in the segment $Q$, $a_5'(Q)a_2$; by relabeling by a cyclic permutation we may assume they are $b_1,b_2,b_3$.
Change the tentative choice to $z=b_2, x=a_3, y=a_3'$.
Then $a_2,b_3$ are in (say) $P$ and $b_1,a_5$ are in $Q$.

This finishes the proof of (**), and the proof of the theorem.
\end{proof}

Counterexample~\ref{ex} naturally rises the following general question:
\begin{problem}
For any $j\geq 4$, relate the connectedness of the $j$-strand of an edge ideal to classical graph theoretic properties of the corresponding graph (or its complement).
\end{problem}

\section{Application to the subadditivity problem}\label{sec:App}
In this section we prove Theorem~\ref{thm:II}.
First we treat the cases in Theorem~\ref{thm:II} where $b=2,3$ and $t_b\in\{2b,2b-1\}$, next we treat the remaining case $t_3=4$.
Let $I=I(G)$ be the edge ideal of a graph $G$. By considering Taylor's resolution of $S/I$, we observe that $t_i\leq 2i$. On the other hand, since the resolution is minimal, we have $t_i\geq i+1$ when $t_i\neq0$. It follows that if $t_i\neq 0$ then
$i+1\leq t_i \leq 2i$ for all $i\geq1$.
\begin{proposition}\label{a=2}
Over any field, if $I=I(G)$ is the edge ideal of a graph $G$, $b\in\{2,3\}$ and $t_b\in\{2b,2b-1\}$, then for all $a\geq0$,
$$t_{a+b}\leq t_a+t_b.$$
\end{proposition}
\begin{proof}
Let $\Delta$ be the simplicial complex such that $I_{\Delta}=I$. If $t_b=2b$ then the known inequality $t_{a+1}\leq t_a +t_1$ implies
$$t_{a+b}\leq t_a+bt_1=t_a+2b=t_a+t_b.$$
Assume $t_b=2b-1$. Let $W\subseteq [n]$ so that $|W|=t_{a+b}$ and $\widetilde{H}_{j}(\Delta[W])\neq 0$ for $j=t_{a+b}-(a+b)-1$.
Denote $N=\Delta[W]$.

We claim that there is a vertex of degree $\geq 2$ in $H=G[W]$.
Else, $H$ would be a disjoint union of edges, i.e. a matching (there are no isolated vertices as $N$ is not acyclic, in particular not a cone).
The Taylor resolution shows that $H$ has at least $a+b$ edges.
Thus, $t_b=2b$, a contradiction.


Let $v$ be a vertex of $H$ of degree $\deg_H(v)\geq2$. Let $x_1,x_2$ be two neighbors of $v$ in $H$. Clearly, $N=(N-v)\cup(N-\{x_1,x_2\})$. Set $\Delta_1=N-v$ and $\Delta_2=N-\{x_1,x_2\}$. We may assume that $\widetilde{H}_{j}(\Delta_1)=0$. For otherwise, we obtain that $\beta_{a+b-1,t_{a+b}-1}(S/I)\neq0$ and so,
$t_{a+b}\leq t_{a+b-1}+1$; for $b=2$ this yields, using (*) for $b=1$,
\[t_{a+2}\leq t_a+t_1+1=t_a+3=t_a+t_2\] as desired, and for $b=3$, after we finish the proof for $b=2$ below, it yields
\[t_{a+3}\leq t_a+t_2+1\leq t_a+5=t_a+t_3.\]

 Using the Mayer-Vietoris sequence for $N=\Delta_1\cup \Delta_2$
 we have $\widetilde{H}_j(\Delta_2)\neq0$ or $\widetilde{H}_{j-1}(\Delta_1\cap\Delta_2)\neq0$.

 If $\widetilde{H}_j(\Delta_2)\neq0$, then $\beta_{a+b-2,t_{a+b}-2}(S/I)\neq0$ and so
 $t_{a+b}\leq t_{a+b-2}+2$.
For $b=2$ this gives $t_{a+2}\leq t_a+2<t_a +t_2$ as desired, and for $b=3$ it gives
$t_{a+3}\leq t_{a+1}+2\leq t_a+t_1+2=t_a +4<t_a+t_3$ as desired.

Finely, if $\widetilde{H}_{j-1}(\Delta_1\cap\Delta_2)\neq0$, then $\beta_{a+b-2,t_{a+b}-3}(S/I)\neq0$, and similarly we obtain for $b=2,3$,
$$t_{a+b}\leq t_{a+b-2}+3 \leq t_a+t_b.$$
This completes the proof.
\end{proof}
Next we prove the remaining case in Theorem~\ref{thm:II}:
\begin{proposition}\label{a=3}
Over any field, if $I=I(G)$ is the edge ideal of a graph $G$ and $t_3=4$, then for all $a\geq0$,
$$t_{a+3}\leq t_a+t_3.$$
\end{proposition}
\begin{proof}

We split the proof according to the possible values of $t_2$.

If $t_2=4$, then combined with $t_3=4$, Theorem~\ref{strand2} (connectivity of the $3$-strand of $I$, equivalently of the $2$-strand of $S/I$) says $\beta_{2+k,4+k}(S/I)=0$ for any $k>0$. Further, by Lemma~\ref{lem:corner} we conclude that $t_k\leq k+1$ for all $k\geq 3$. In particular, if $t_{a+3}\neq0$, then $t_{a+3}\leq a+4\leq t_a+t_3$ as desired.

Thus we may assume $t_2=3$. By Proposition~\ref{a=2} we may also assume $a\geq3$.
Let $\Delta$ be the simplicial complex such that $I_{\Delta}=I$ and let $W\subseteq [n]$ so that $|W|=t_{a+3}$ and $\widetilde{H}_{t_{a+3}-(a+3)-1}(\Delta[W])\neq0$.

First, we claim that there exists a vertex $v$ in $H=G[W]$ such that $\deg_H(v)\geq3$.

Assume on contrary, that $\deg_H(v)\leq2$ for all $v$ in $H$.
Then $H$ is the disjoint union of cycles and paths, and as $\Delta[W]$ is not acyclic it has no isolated vertices. Further, as $t_2<4$, $H$ has no two disjoint edges which form an induced subgraph. Thus, $H$ is either an induced $t$-cycle $C_t$, where $3\leq t\leq 5$, or a path with at most three edges.
On the other hand,
since $\beta_{a+3,t_{a+3}}(S/I(H))\neq0$ and $a\geq3$, it follows from the Taylor resolution that there are at least 6 edges in $H$, a contradiction.


Let $v\in W$ with $\deg_H(v)\geq3$. Let $x_1,x_2,x_3$ be three neighbors of $v$ in $H$. Clearly, for $N=\Delta[W]$,  $N=(N-v)\cup(N-\{x_1,x_2,x_3\})$. Set $\Delta_1=N-v$ and $\Delta_2=N-\{x_1,x_2,x_3\}$. We may assume that $\widetilde{H}_{j}(\Delta_1)\neq0$. For otherwise, we obtain that $\beta_{a+2,t_{a+3}-1}(S/I)\neq0$ and so $$t_{a+3}\leq t_{a+2}+1\leq t_a+t_2+1=t_a+4=t_a+t_3$$ as desired.
Using the Mayer-Vietoris sequence we have $\widetilde{H}_j(\Delta_2)\neq0$ or $\widetilde{H}_{j-1}(\Delta_1\cap\Delta_2)\neq0$. If $\widetilde{H}_j(\Delta_2)\neq0$, then $\beta_{a,t_{a+3}-3}(S/I)\neq0$, and so $$t_{a+3}\leq t_{a}+3 < t_a+t_3.$$
Finally, if $\widetilde{H}_{j-1}(\Delta_1\cap\Delta_2)\neq0$, then $\beta_{a,t_{a+3}-4}(S/I)\neq0$, and so $$t_{a+3}\leq t_a+4=t_a+t_3.$$
\end{proof}

\textbf{Acknowledgements.} The authors thank Aldo Conca for helpful comments and discussion.

\end{document}